\documentclass{article}
\usepackage{arxiv}
\usepackage{float}
\usepackage[utf8]{inputenc} 
\usepackage[T1]{fontenc}    
\usepackage{hyperref}       
\usepackage{url}            
\usepackage{booktabs}       
\usepackage{amsfonts}       
\usepackage{nicefrac}       
\usepackage{microtype}      
\usepackage{lipsum}
\usepackage{enumitem}
\usepackage{graphicx}
\usepackage{amsmath}
\usepackage{amsthm,amssymb}
\newtheorem{thm}{Theorem}
\newtheorem{lem}{Lemma}

\title{Eigenvalues of some classes of signed complete graphs}
\author{Prajnanaswaroopa S\\
sntrm4@rediffmail.com}

\date{sntrm4@rediffmail.com}

\begin{document}
\maketitle
\section*{Abstract}
In this work, we discuss some properties of the eigenvalues of some classes of signed complete graphs. We also obtain the form of characteristic polynomial for these graphs.
\section*{Introduction}
Signed graphs was introduced by Frank Harary \cite{HAR}. Signed graphs are becoming highly useful tools to analyze several networks in real life, typically social networks \cite{GAL}. A signed graph is a simple loopless graph with a function defined from the set of edges to the set $\{-1,1\}$. Spectrum of signed graphs have been discussed in \cite{ZAS1}, \cite{ZAS2}. Spectrum of complete signed graphs have been discussed in some length in \cite{AKB}. . Here, we find the eigenvalues of the adjacency matrix associated to the signed complete graph $G'$, where the negative edges induce a graph $G$ of order $n$ which consist of a union of $k$ cliques of order $h$ such that we have a $h-p$ clique common to all the $k$ cliques, and each of the $p$ vertices of each of the $k$ cliques are disjoint. In other words, if $G$ be the induced graph formed by the negative edges, and, if we label the vertices of the disjoint cliques as $v_{ij}$, $i\in\{1,2,\ldots,p\}$ and $j\in\{1,2,\ldots,k\}$, then the vertices $v_{ml}$, where, $m\in\{h-p+1,h-p+2,\ldots,h\}$ and $l\in\{1,2,\ldots,k\}$ form a clique in the graph $G$. We can call the induced graph as $G$ with parameters $n,h,p$. The adjacency matrix of such a graph using a suitable labelling of its vertices can be given by: 
$$\begin{pmatrix}
K_p&O&\ldots&O&X_p\\O&K_p&\ldots&O&X_p\\O&O&\ddots(k-1)-times&O&X_p\\O&O&\ldots&K_p&X_p\\X&O&\ldots&O&K_h
\end{pmatrix}$$
where $O$ is the zero matrix and $X_p$ are the first $p$ rows of a matrix $X$ given by 
$$X=\begin{pmatrix}
J&O&\ldots&O\\J&O&\ldots&O\\\vdots&\vdots&(k-1)-times&\vdots\\J&O&\cdots&O\\J&O&\ldots&O\end{pmatrix}$$
with $J_{p,(h-p)}$ being the all ones matrix of order $p\times (h-p)$ given by
$$J=\begin{pmatrix}
1&1&\ldots&1\\1&1&\ldots&1\\\vdots&\vdots&\vdots&\vdots\\1&1&\ldots&1
\end{pmatrix}
$$
As the adjacency matrix for a signed complete graph with negative edges inducing a graph $G$ is the same as the Siedel adjacency of the graph $G$; and, since the Siedel adjacency matrix of a graph with adjacency matrix $A$ of order $n$ is defined as $J_n-I_n-2A$, therefore, we get that, the matrix of which we wish to find the spectrum is 
$$\begin{pmatrix}
-K_{p}&J&\ldots&J&X'_{1}\\J&-K_{p}&\ldots&J&X'_{2}\\J&J&\ddots(k-1)-times&J&X'_{3}\\J&J&\ldots&-K_p&X'_{4}\\X'_{5}&J&\ldots&J&-K_h
\end{pmatrix}$$
with $X'_{i}$ being the $i$-th $p$ rows of the matrix $X'$ given in block form by:
$$X'=\begin{pmatrix}
-J_{(k-1)p,h-p}&J_{(k-1)p,p}\end{pmatrix}$$. We denote the main matrix by $S(G)$ with respect to parameters $n,h,p$.
\section*{Main Theorems}
\begin{lem}
The eigenvalues of the matrix $X=aI_n+bJ_n$ are given by $bn+a$ with multiplicity $1$ and $a$ with multiplicity $n-1$.
\end{lem}
\begin{proof}
The proof is quite straightforward. We observe that $bJ_n$ has rank $1$, and hence a single non-zero eigenvalue. We also observe that $(1\ 1\ 1\ldots\ 1^T$ is an eigenvector of $bJ_n$ with corresponding eigenvalue $bn$. The eigenvectors of $bJ_n$ are also eigenvectors of $X$, as $aI_n$ is a scalar matrix. Therefore, as the only eigenvalues of $aI_n$ with respect to any of its eigenvectors are $a$, the eigenvalues of $X$ are $a$ with multiplicity $n-1$ and $bn+a$  with multiplicity $1$. For more explicit clarity, we note that the eigenvectors of $X$ are $(1\ 1\ \ldots\ 1)^T$ with eigenvalue $nb+a$, and $ (1\ 0\ 0\ldots\ -1)^T,(0\ 1\ 0\ldots\ -1)^T\ldots,(0\ 0\ 0\ldots\ 0\ -1)^T$ each with eigenvalue $a$
\end{proof}
\begin{thm}
If $A$ is a $k$ order square matrix having constant row sum $r$ and having the same eigenvectors as $aI+bJ$. Then the eigenvalues of the matrix $M$ defined by
$$\begin{pmatrix}A&bJ_k&\ldots&\ bJ_k\\bJ_k&A&\ldots&bJ_k\\\vdots&\vdots&(n-times)&\vdots\\bJ_k&\ldots&\ldots&A\end{pmatrix}$$ are given by $r+bk(n-1)$ with multiplicity $1$, $r-bk$ with multiplicity $n-1$ and $d$ with multiplicity $k(n-1)$, where $d$ is the eigenvalue of $A$ with respect to the eigenvectors other than $(1\ 1\ 1\ldots\ 1)^T$.
\end{thm}
\begin{proof}
Taking the previous theorem as an inspiration, we can construct eigenvectors for $M$ as follows:
Let $\vec{j}_i$ denote the all ones vector of order $i$. Let the eigenvectors of $A$ except $\vec{j}_k$ be labelled as $e_1, e_2, \ldots, e_{k-1}$. Then, we have the eigenvectors of $M$ to be $\vec{j}_{kn},(\vec{j}_k\ 0\ \ldots\ -\vec{j}_k)^T,(0\ \vec{j}_k\ 0\ldots -\vec{j}_k)^T,\ldots,(0\ 0\ \ldots0\ -vec{j}_k)^T,(e_1\ 0\ 0\ \ldots0)^T,(e_2\ 0\ \ldots\ 0)^T,\ldots, (e_{k-1}\ 0\ \ldots\ 0)^T, (0\ e_1\ 0\ldots\ 0)^T,\ldots, (0\ 0\ \ldots\ e_{k-1})^T$. The corresponding eigenvalues would then be $r+(bkn-bk)=r+bk(n-1)$ with multiplicity $1$ (for eigenvector $\vec{j}_{kn}$), $r-bk$ with multiplicity $n-1$ (for eigenvectors
$(\vec{j}_k\ 0\ \ldots\ -\vec{j}_k)^T,(0\ \vec{j}_k\ 0\ldots -\vec{j}_k)^T,\ldots,(0\ 0\ \ldots0\ \vec{j}_k\ -vec{j}_k)^T$) and lastly the eigenvalue of $A$ corresponding to the eigenvectors $e_i$ with multiplicity $k(n-1)$. The vectors given above are actually eigenvectors can be easily verified by multiplication and by using the properties of $A$ and $bJ_k$.
\end{proof}

\begin{lem}
We have $X'J_h=J_hX'^T=(2p-h)J_h$, where $$X'=\begin{pmatrix}
-J_{(k-1)p,h-p}&J_{(k-1)p,p}\end{pmatrix}$$

\end{lem}
\begin{proof}
As $X'$ is a matrix having constant row and columns sums, the entries of the matrix $X'J_h$ will consist only be the row sum of $X'$. This is because $J_h$ consists of only $1$s. Similarly, the entries of the matrix $J_hX'^{T}$ consist only of the column sum of $X'^{T}$, which is equal to the row sum of $X'$. As any row of $X'$ equals $(-1)$ $h-p$ times and $1$ $p$ times, the row sum would be $p-(h-p)=2p-h$, from which the lemma follows.  
\end{proof}
\begin{lem}
Let $A$, $D$ be square matrices of arbitrary orders. The determinant of the block matrix
$$M=\begin{pmatrix}
A&B\\C&D\end{pmatrix}$$ with invertible matrix $D$
is given by $|M|=|D||A-CD^{-1}B|=\frac{||D|A-C\cdot adj(D)B|}{|D|^{n-1}}$, where $adj(D)$ is the adjugate (or adjoint) of $D$.
\end{lem}
\begin{proof}
Proof is by using the Schur complement of the matrix $D$ in the matrix $M$. In other words, when we perform Gaussian elimination of the block matrix $M$, knowing that $D$ is invertible, we obtain the reduced matrix as $$M=\begin{pmatrix} A-BD^{-1}C&O\\D^{-1}C&I\end{pmatrix}$$, where $I$ is the identity matrix of appropriate order, that is, the same order as that of $D$. Knowing that the determinant of the block matrix $$M=\begin{pmatrix} A-BD^{-1}C&O\\D^{-1}C&I\end{pmatrix}$$ is just the product of the diagonals \cite{SIL}, which, in this case, is $|D|\cdot|A-BD^{-1}C|$. The expression involving adjuagte follows by noting that $A^{-1}=\frac{adj(A)}{|A|}$. 
\end{proof}
\begin{lem}
If $A_n$ is the adjacency matrix of the complete graph $K_n$, then adjugate of $M=-A_n-\lambda I_n$ has the form 
$$\begin{pmatrix}
C_p(n-1)&(\lambda-1)^{n-2}&\ldots&(1-\lambda)^{n-2}\\
(1-\lambda)^{n-2}&C_p(n-1)&\ldots&(1-\lambda)^{n-2}\\\vdots&\vdots&\ddots&\vdots\\(1-\lambda)^{n-2}&\ldots&ldots&C_p(n-1)
\end{pmatrix}$$
, where $C_p(n)=(1-\lambda)^{n-1}(1-n-\lambda)$ is the characteristic polynomial of $M$ (or negative of adjacency matrix of the complete graph of order $n$).
\end{lem}
\begin{proof}
As $A_n=(1-\lambda)I_n-J_n$ is invertible, we again use the property that $adj(A)=|A|\cdot A^{-1}$. To calculate the inverse of $M$, we use Sherman-Morrison formula \cite{SHE}. By the formula, we have $(X+uv^T)^{-1}=X^{-1}-\frac{X^{-1}uv^TX^{-1}}{1+v^TX^{-1}u}$, where $u$ and $v$ are vectors and $1+v^TA^{-1}u\neq0$. Here, we can take $u=(1\ 1\ \ldots(n-times)\ 1)^T$ and $v=(-1\ -1\ \ldots(n-times)\ -1)^T$ so that $uv^T=-J_n$ and $X=(1-\lambda)I_n$. Then, we get
\[A_n^{-1}=\frac1{1-\lambda}I_n-\frac{\frac1{1-\lambda}I_nuv^T\frac1{1-\lambda}I_n}{1+\frac{-n}{1-\lambda}}\]
\[=\frac1{1-\lambda}I_n-\frac{\frac1{(1-\lambda)^2}(-J)}{\frac{(1-\lambda)-n}{1-\lambda}}\]
\[=\frac1{(1-\lambda)(1-\lambda-n)}(1-\lambda-n)I_n+J_n\]
This implies that the adjugate then becomes $|A_n|\cdot A_n^{-1}$
$=C_p(n)\frac1{(1-\lambda)(1-\lambda-n)}(1-\lambda-n)I_n+J_n\cdot $
$=(1-\lambda)^{n-2}(1-\lambda-n)I_n+(1-\lambda)^{n-2}J_n=(1-\lambda)^{n-2}((2-\lambda-n)-1)I_n+(1-\lambda)^{n-2}J_n$
$=(C_p(n-1)-(1-\lambda)^{n-2})I_n+(1-\lambda)^{n-2}$. This matrix, when expanded, at once gives the desired result.
\end{proof}
\begin{lem}
If $K_h$ denotes the adjacency matrix of the complete graph on $h$ vertices, then we have $X'\cdot adj(-K_h-\lambda I_h)\cdot X'^T=(((C_p(h-1)-(1-\lambda)^{h-2}))h+((2p-h)^2)(1-\lambda)^{h-2})J_{n-h}$, where $$X'=\begin{pmatrix}
-J_{(k-1)p,h-p}&J_{(k-1)p,p}\end{pmatrix}$$ as before.
\end{lem}
\begin{proof}
The proof is straight-forward multiplication and using Lemmas $2$ and $4$.
\end{proof}
\begin{proof}
In this case, the matrix has the form
$$\begin{pmatrix}
-K_{n-h}&X\\X^{T}&-K_h
\end{pmatrix}$$
with $$X=\begin{pmatrix}v&J_{n-h,h-1}
\end{pmatrix}$$ and $v$ is the vector $(-1,-1,\ldots,-1)^T$.First, let us find the characteristic polynomial.
\end{proof}
\begin{thm}
The spectrum of the matrix $S(G)$ with parameters $n,h,p$ is given by the roots of the polynomial $F(\lambda)=(1-2p-\lambda)^{\frac{n-h}{p}-1}(1-\lambda)^{n-2-\frac{n-h}{p}}s$, where $s= - \lambda^3- (2h - n + 2p - 3)\lambda^2- (2h^2 - 2(h - 1)n + 2(h - 2)p - 4h + 3)\lambda + 2h^2 - (2h - 1)n - 2(2h^2 - 2hn - h + 1)p  - 2h+4(h - n)p^2 + 1$. In particular, it has eigenvalue of $1$ with multiplicity at least $n-h-\frac{n-h}{p}$, and $2p-1$ with multiplicity at least $\frac{n-h}{p}-1$.
\end{thm}
\begin{proof}
The Siedel matrix $S(G)$ is given by:
$$\begin{pmatrix}
-K_{p}&J&\ldots&J&X'_{1}\\J&-K_{p}&\ldots&J&X'_{2}\\J&J&\ddots(k-1)-times&J&X'_{3}\\J&J&\ldots&-K_p&X'_{4}\\X'_{5}&J&\ldots&J&-K_h
\end{pmatrix}$$, with $X'_{i}$ being the $i$-th $p$ rows of the matrix $X'$ given in block form by:
$$X'=\begin{pmatrix}
-J_{(k-1)p,h-p}&J_{(k-1)p,p}\end{pmatrix}$$. We proceed to calculate the characteristic polynomial of the matrix $S(G)$. This is nothing but the determinant of the matrix $S(G)-\lambda I_n$. In matrix form, this is
$$\begin{pmatrix}
-K_{p}-\lambda I_p&J&\ldots&J&X'_{1}\\J&-K_{p}-\lambda I_p&\ldots&J&X'_{2}\\J&J&\ddots(k-1)-times&J&X'_{3}\\J&J&\ldots&-K_p-\lambda I_p&X'_{4}\\X'_{5}&J&\ldots&J&-K_h-\lambda I_h
\end{pmatrix}$$.
By using Lemma $3$, the determinant can be written as:
$\frac{|-K_h-\lambda I_h|M-X'\cdot adj(-K_h-\lambda I_h)X'^T|}{|-K_h-\lambda I_h|^{n-1}}$, where $M$ is the block matrix of the first $n-h$ rows and columns given by 
$$\begin{pmatrix}
-K_{p}-\lambda I_p&J&\ldots&J\\J&-K_{p}-\lambda I_p&\ldots&J\\J&J&\ddots(k-1)-times&J\\J&J&\ldots&-K_p-\lambda I_p
\end{pmatrix}$$.
Taking cognizance of the fact that $|-K_h-\lambda I_h|=C_p(h)$ and $X'\cdot adj(-K_h-lambda I_h)\cdot X'^T=(1-\lambda)^{n-2}(1-\lambda-n)I_n+(1-\lambda)^{n-2}J_n$ from Lemma $5$, we get the determinant as $\frac{|C_p(h)M-((1-\lambda)^{n-2}(1-\lambda-n)I_n+(1-\lambda)^{n-2}J_n)|}{(C_p(h))^{n-h-1}}$. We take $Y=(1-\lambda)^{n-2}(1-\lambda-n)I_n+(1-\lambda)^{n-2}$. Then, the determinant becomes, in block form:

\begin{center}
\[\frac{1}{(C_p(h))^{n-h-1}}\left|\begin{smallmatrix}
C_p(h)[-K_{p}-\lambda I_p]-YJ_p&J_p(C_p(h)-Y)&\ldots&J_p(C_p(h)-Y)\\J_p(C_p(h)-Y)&C_p(h)[-K_{p}-\lambda I_p-YJ_p]&\ldots&J_p(C_p(h)-Y)\\J_p(C_p(h)-Y)&J_p(C_p(h)-Y)&\ddots(k-1)-times&J_p(C_p(h)-Y)\\J_p(C_p(h)-Y)&J_p(C_p(h)-Y)&\ldots&C_p(h)[-K_p-\lambda I_p-YJ_p]
\end{smallmatrix}\right|\]
\end{center}
The above matrix has a similar form to the one in Theorem $1$, with $A=Cp(h)[-K_p-\lambda I_p]-YJ_p$ and $b=Cp(h)-Y$. Therefore, as the eigenvalues are $(C_p(h)(n-h-2p+1-x)+Y(h-n))$ with multiplicity $1$, $(C_p(h)(1-2p-x))$ with multiplicity $\frac{n-h}{p}$ the determinant will be equal to $\frac1{(C_p(h))^{n-h-1}}[(C_p(h)(n-h-2p+1-x)+Y(h-n))(C_p(h)(1-2p-x))^{\frac{n-h}{p}-1}(C_p(h)(1-x))^{n-h-\frac{n-h}{p}}]$. Simplifying the expression using the form of $C_p(h)=(1-\lambda)^{h-1}(1-\lambda-h)$, we get $F(\lambda)=(1-2p-\lambda)^{\frac{n-h}{p}-1}(1-\lambda)^{n-2-\frac{n-h}{p}}s$, where $s= - \lambda^3- (2h - n + 2p - 3)\lambda^2- (2h^2 - 2(h - 1)n + 2(h - 2)p - 4h + 3)\lambda + 2h^2 - (2h - 1)n - 2(2h^2 - 2hn - h + 1)p  - 2h+4(h - n)p^2 + 1$. The expression $F(\lambda)$ is therefore the characteristic polynomial of $S(G)$. Therefore, roots of the cubic polynomial $s$ will fully determine the spectrum of $G$, as the eigenvalues $1$ and $2p-1$ are already known with their minimum multiplicities $n-2-\frac{n-h}{p}$ and $\frac{n-h}{p}-1$ from the expression.
\end{proof}
\section*{Conclusion}
In this paper, we have used the block matrix technique, Sherman-Morrison formula and eigenvector reconstruction method to compute the spectrum and characteristic polynomial of certain signed complete graphs. The method as such can have long-lasting applications in spectral graph theory and spectrum of the signed complete graphs can also be used further for various applications.

\end{document}